\documentclass[a4paper,12pt, reqno]{amsart}
\usepackage{a4wide}
\usepackage{amsthm}
 \newtheorem{theorem}{Theorem}[section]
 
 \newtheorem{lemma}[theorem]{Lemma}

 \newcommand{\mbb}{\mathbb}
 \newcommand{\R}{\mbb{R}}

\begin{document}

\title{Simultaneous Diophantine approximation on lines with prime constraints}

\author{Stephan Baier}
\author{Anish Ghosh} 
\address{S. Baier, A. Ghosh, School of Mathematics, Tata Institute of Fundamental Research, Homi Bhabha Road, Colaba, Mumbai, India 400005}
\thanks{Baier is supported by an ISF-UGC grant. Ghosh is supported by an ISF-UGC grant. This material is based upon work supported by the National Science Foundation under Grant No. 0932078 000 while Ghosh was in residence at the Mathematical Sciences Research Institute in Berkeley, California, during the Spring 2015 semester.}

\email{sbaier@math.tifr.res.in} 
\email{ghosh@math.tifr.res.in}

\title{Restricted simultaneous Diophantine approximation}

\begin{abstract}
We study the problem of Diophantine approximation on lines in $\R^d$ under certain primality restrictions. 
\end{abstract}


\maketitle

{\small
\tableofcontents

\section{Introduction}

The subject of \emph{metric Diophantine approximation on manifolds} studies Diophantine approximation of typical points on submanifolds in $\R^d$ by rational points in $\R^d$. This subject has received considerable attention in the last two decades, leading to dramatic progress using methods arising from the ergodic theory of flows on homogeneous spaces as well as analytic methods. If one puts further restrictions on the approximating rationals, then the situation is much less understood. A very natural class of problems arises by imposing primality restrictions on the approximating rationals. In this paper, we study the problem of Diophantine approximation for vectors on lines in $\R^d$ with additional primality restrictions. Thus we combine the themes of simultaneous metric Diophantine approximation \emph{on affine subspaces}, and Diophantine approximation with restrictions. Both problems have their own substantial complications and have separately received considerable attention, cf. \cite{BBDD, Ghosh, Kl1} for Diophantine approximation on affine subspaces, and \cite{HarmK, Harm, Harm-book, HeJi, Jones} for Diophantine approximation with primality constraints.  The only previous works on the combined theme that we are aware of are the work of Harman-Jones \cite{HaJo} regarding Diophantine approximation on curves in $\R^2$ with constraints, and our previous work \cite{BG} where we addressed the problem of Diophantine approximation with constraints on lines in $\R^2$. The main result of the present paper generalises \cite{BG} to arbitrary dimensions and is the first such result in this generality. While the broad strategy in the present paper is similar to that of \cite{BG}, the greater generality makes the problem significantly more complicated. 
In studying Diophantine approximation on lines, and more generally, affine subspaces, it is natural and indeed imperative to impose some Diophantine condition on the line or subspace itself. In \cite{BG}, we had assumed that the slope of the line in $\R^2$ is irrational. In higher dimensions, the lack of a suitable continued fraction algorithm makes it necessary to replace irrationality with a suitable Diophantine condition which we now introduce. Let $\|~\|$ denote the distance to the nearest integer of a real number,  $\|~\|_{\infty}$ denote the supremum norm of a vector and for vectors ${\bf v}, {\bf c} \in \R^d$, denote by ${\bf v} \cdot {\bf c}=v_1c_1+...+v_dc_d$ the inner product of ${\bf v}$ and ${\bf c}$.  Recall that ${\bf c} \in \R^d$ is called $k$-Diophantine (${\bf c} \in D_{k}(\R^d)$) if there exists a constant $C>0$ such that
\begin{equation} \label{condi}
||{\bf v}\cdot {\bf c}||>\frac{C}{||{\bf v}||_{\infty}^k} \quad \mbox{for every } {\bf v}\in \mathbb{Z}^d\setminus \{{\bf 0}\}.
\end{equation}
Our main result is:

\begin{theorem} \label{super}
Let $d$ be a positive integer and $k\ge d$ be a positive real number. Define
\begin{equation} \label{gammadef}
\gamma_{d,k}:=\frac{1}{d(3k+2)}
\end{equation}
and suppose that $0<\varepsilon<\gamma_{d,k}$.
Let $c_1,...,c_d$ be positive irrational numbers such that the vector ${\bf c}=(c_1,...,c_d)$ is $k$-Diophantine. Then for almost all positive real $\alpha$, with respect to the Lebesgue measure, there are infinitely many $(d+2)$-tuples $(p,q_1,...,q_d,r)$ with $p$ and $r$ prime and $q_1,...,q_d$ positive integers such that simultaneously
\begin{equation} \label{simultan}
\begin{split}
0 & <p\alpha-r<p^{-\gamma_{d,k}+\varepsilon},\\
0 & <pc_i\alpha-q_i<p^{-\gamma_{d,k}+\varepsilon} \mbox{ for all } i\in \{1,...,d\}.
\end{split}
\end{equation}
\end{theorem}

\noindent {\bf Remarks}:
\begin{enumerate}

\item It is well known that $D_{d}(\R^d)$ is a nonempty set of zero Lebesgue measure and full Hausdorff dimension. These comprise the set of \emph{badly approximable} vectors. Moreover, $D_{k}(\R^d)$ has full measure whenever $k > d$, see \cite{Cassels} for example.\\

\item In \cite{BG}, we proved the analogue of Theorem \ref{super} for lines in $\R^2$ under the assumption that the slope $c$ of the line is irrational. In fact, what was used was the following Diophantine property of irrational numbers:  There exists an infinite set $\mathcal{S}$ of integers  and a positive constant $D>0$ such that
\begin{equation} \label{nustatement}
\min\limits_{\substack{{\bf v}\in \mathbb{Z}^d\setminus \{{\bf 0}\}\\ ||{\bf v}||_{\infty}\le N}} || {\bf v} \cdot {\bf c}|| \ge N^{-D} \quad \mbox{if} \quad N\in \mathcal{S}.
\end{equation}

\noindent For $d=1$ this holds for $D=1$ and $\mathcal{S}$ being the set of numbers $[N/2]$, where the $N$'s are the denominators in the continued fraction approximants of $c=c_1$. This can be seen by approximating $c$ by its continued fraction approximant $a/N$ and using the fact that  $|c-a/N|\le 1/N^2$. In the case $d=1$, a sequence $\mathcal{S}$ with the above property can also be constructed using the Dirichlet approximation theorem and the condition that $c$ is irrational. There is a $d$-dimensional version of Dirichlet's approximation theorem, but for $d\ge 2$, a statement like \eqref{nustatement} does not follow from it. It is likely that the conclusion of Theorem \ref{super} would also hold in arbitrary dimension by assuming (\ref{nustatement}) rather than the Diophantine condition we have assumed, with perhaps a different exponent. We note however, that for $d = 1$, choosing $k=1$, we recover the exponent in \cite{BG}. The condition (\ref{nustatement}) is also an interesting Diophantine property, and can be shown to include the class of \emph{nonsingular} vectors.\\

\item It is an interesting problem to consider analogues of Theorem \ref{super} for affine subspaces of lower codimension, and indeed for manifolds not contained in affine subspaces, the \emph{non degenerate} manifolds. We will consider these in a forthcoming work. 

\end{enumerate}

\subsection*{Acknowledgements} S. Baier wishes to thank the Tata Institute of Fundamental Research in Mumbai (India) for its warm hospitality and excellent working conditions. This work was completed while Ghosh was a visiting professor at the Technion, Israel Institute of Technology and a member at MSRI Berkeley. The hospitality of both institutions is gratefully acknowledged.

\section{A Metrical approach}
Our method is based on the following lemma in \cite{HaJo}. 

\begin{lemma} \cite[Lemma 1]{HaJo} \label{Hajolemma} Let $A$ and $B$ be reals with $B > A > 0$. Let $F_N(\alpha)$ be a nonnegative
valued function of N (an integer) and $\alpha$ (a real variable), and $G_N$, $V_N$
functions of N such that:\\

(i) $G_N\rightarrow \infty$\quad as $N\rightarrow \infty$,\\ 

(ii) $V_N=o\left(G_N\right)$ \quad as $N\rightarrow \infty$,\\

(iii) for all $a$, $b$ with $A\le a<b\le B$ we have
$$
\limsup\limits_{N\rightarrow \infty} \int\limits_a^b \frac{F_N(\alpha)}{G_N} \ d\alpha \ge b-a
$$

(iv) there is a positive constant K such that, for any measurable set $\mathcal{C}\subseteq [A,B]$, 
$$
\int\limits_{\mathcal{C}} F_N(\alpha) \ d\alpha \le KG_N\lambda(\mathcal{C})+V_N.
$$

Then, for almost all $\alpha \in [A,B]$, we have 
\begin{equation} \label{endresult}
\limsup\limits_{N\rightarrow \infty} \frac{F_N(\alpha)}{G_N} \ge 1. 
\end{equation}
\end{lemma}

In our application, $F_N(\alpha)$ will be the number of solutions to \eqref{simultan} with $p<N$. Further, for given $0<A<B$ we set
$$
G_N=G_N(A,B)=\frac{A^2}{B^2} \cdot \frac{\min(c_1,...,c_d,d)^{d-1}}{2^{d+1}} N^{1-(d+1)(\gamma_{d,k}-\varepsilon)}(\log N)^{-2}
$$
We will prove

\begin{theorem} \label{Theo}
The following holds for every natural number $N$.\\

(i) Let $0 < A < B$. Then for all $a,b$ with $A\le a<b\le B$ we have
\begin{equation} \label{LOL}
\int\limits_a^b F_N(\alpha) d\alpha \ge (b-a)G_N(A,B)(1+o(1))
\end{equation}
if $N\in \mathcal{S}$ and $N\rightarrow \infty$.\\

(ii) Let $0<A<B$ and $\varepsilon>0$. Then there exists a constant $K=K(A,B,\varepsilon)$ such that, for $\alpha\in [A,B]$, we have 
$$
F_N(\alpha)\le KG_N(A,B)+J_N(\alpha)
$$
with
$$
\int\limits_A^B \left|J_N(\alpha)\right| d\alpha=o\left(G_N(A,B)\right) \quad \mbox{as } N\rightarrow \infty
$$
if $N\in \mathcal{S}$ and $N\rightarrow \infty$. 
\end{theorem}

Theorem \ref{Theo}(i) corresponds to Lemma 2, Theorem \ref{Theo}(ii) to Lemma 3 in \cite{HaJo}. Now it follows that conditions (i) to (iv) in Lemma \ref{Hajolemma} are satisfied 
for
$$
V_N=\int\limits_A^B \left|J_N(\alpha)\right| d\alpha.
$$
Now the claim in Theorem \ref{super} follows from \eqref{endresult}.  

\section{Proof of Theorem \ref{Theo}(\MakeLowercase{i})}

\subsection{Reduction to a counting problem}
To prove Theorem \ref{Theo}(i), we broadly follow the approach in section 3 of \cite{HaJo} . However, we use exponential sum estimates instead of zero density estimates for the Riemann zeta function since they turn out to be more suitable for our purposes. This is the content of the next subsection.\\


Throughout the sequel, we denote by $\mathbb{P}$ the set of primes. Let 
$$
\mathcal{B}_p=\bigcup_{\substack{r\in \mathbb{P}\\ q_1,...,q_d\in \mathbb{N}}} \left[\left.\frac{r}{p},\frac{r+\tilde\eta}{p}\right)\right.\cap \left[\left. \frac{1}{c_1}\cdot \frac{q_1}{p},\frac{1}{c_1}\cdot \frac{q_1+\tilde\eta}{p}\right)\right.\cap\cdots \cap \left[\left. \frac{1}{c_d}\cdot \frac{q_d}{p},\frac{1}{c_d}\cdot \frac{q_d+\tilde\eta}{p}\right)\right.\cap [a,b],
$$
where $\tilde\eta=p^{\varepsilon-\gamma_{d,k}}$. Then
\begin{equation} \label{integral}
\int\limits_{a}^{b} F_N(\alpha)d\alpha =\sum\limits_{\substack{p \in \mathbb{P}\\ p\le N}} \lambda(\mathcal{B}_p),
\end{equation}

\noindent where $\lambda$ is Lebesgue measure. Set
\begin{equation} \label{mudef}
\mu:=(a+b)/(2a).
\end{equation} 

\noindent Our strategy is to split the interval $[1,N]$ into subintervals $[P,P\mu]$ and sum up over the $P$'s in the end. Accordingly, we restrict $p$ to the interval $P\le p<P\mu$ with $P\mu\le N$. We then obtain a lower bound for \eqref{integral} by replacing $\tilde\eta$ with 
\begin{equation} \label{etadef}
\eta=(\mu P)^{\varepsilon-\gamma_{d,k}}.
\end{equation}
Clearly, $\tilde\eta\ge\eta$ if $P\le p<P\mu$. 

\noindent We note that if 
$$
\frac{r}{p}\le  \frac{1}{c_i}\cdot \frac{q_i}{p} \le \frac{r+\eta/2}{p} \mbox{ for } i=1,...,d,
$$
then 
$$
\lambda\left(\left[\left.\frac{r}{p},\frac{r+\eta}{p}\right)\right. \cap \left[\left.\frac{1}{c_1}\cdot \frac{q_1}{p},\frac{1}{c_1}\cdot \frac{q_1+\eta}{p}\right)\right. \cap \cdots \cap 
\left[\left. \frac{1}{c_d}\cdot \frac{q_d}{p},\frac{1}{c_d}\cdot \frac{q_d+\eta}{p}\right)\right.\right)\ge \nu,
$$
where
\begin{equation} \label{nudef}
\nu:=\frac{\eta}{\mu P}\min \left(\frac{1}{2}, \frac{1}{c_1},...,\frac{1}{c_d}\right)=(\mu P)^{-1-\gamma_{d,k}+\varepsilon}\min \left(\frac{1}{2}, \frac{1}{c_1},...,\frac{1}{c_d}\right).
\end{equation}
Also, for all $p\in [P,P \mu)$,
$$
Pa\mu\le r\le bP \Longrightarrow a\le \frac{r}{p}\le b,
$$
and $r$ here runs over the primes in an interval of length $\frac{b-a}{2}P$. We thus have 
\begin{equation} \label{key}
\sum\limits_{P\le p< \mu P} \lambda(\mathcal{B}_p) \ge \nu N(P),
\end{equation}
where $N(P)$ counts the number of solutions $(p,q_1,...,q_d,r)\in \mathbb{P} \times \mathbb{Z}^d\times \mathbb{P}$ to 
$$
q_i\in \left[ \left. c_ir, c_ir+\delta\right.\right)  \mbox{ for } i=1,...,d, \quad P\le p< P \mu, \quad Pa\mu\le r\le bP,  
$$
where
\begin{equation} \label{alphadef}
\delta:=\frac{\min\{c_1,...,c_d\}\eta}{2}=\frac{\min\{c_1,...,c_d\}}{2(\mu P)^{\gamma_{d,k}-\varepsilon}}.
\end{equation}

\noindent Note that in contrast to the problem considered by Harman and Jones, the conditions on $p$ and  $q$ are here independent, which simplifies matters to some extent. By the prime number theorem, the number $R(P)$ of prime solutions to
$$
P\le p < P \mu 
$$
satisfies 
\begin{equation} \label{RP}
R(P) \sim (\mu-1)P(\log 2P)^{-1}  \mbox{ as } P \rightarrow \infty.
\end{equation}
It remains to count the number of solutions $(q_1,...,q_d,r)\in \mathbb{N}\times \mathbb{P}$ to
$$
q_i\in \left[ \left. c_ir, c_ir+\delta\right.\right) \mbox{ for } i=1,...,d, \quad Pa\mu\le r\le bP,
$$ 
which equals
$$
S(P):=\sum\limits_{\substack{Pa\mu \le r \le bP\\ r \ prime}} \prod\limits_{i=1}^d \left([-c_ir]-[-(c_ir+\delta)]\right).
$$
Let
\begin{equation} \label{TPdef}
T(P):=\sum\limits_{Pa\mu \le n \le bP} \prod\limits_{i=1}^d \left([-c_in]-[-(c_in+\delta)]\right)\Lambda(n).
\end{equation}
We aim to show that 
\begin{equation} \label{TP}
T(P) = \delta^d(b-a\mu)P(1+o(1)) + O\left(N^{1-d\gamma_{d,k}+\varepsilon/2}\right)  \quad \mbox{if } P\mu\le N.
\end{equation}
As usual, from \eqref{TP}, it follows that
$$
S(P) = \delta^d(b-a\mu)P (\log 2P)^{-1}(1+o(1)) + O\left(N^{1-d\gamma_{d,k}+\varepsilon/2}\right) \quad \mbox{if } P\mu\le N,
$$
which together with \eqref{RP}  gives
$$
N(P) = R(P)S(P)= \delta^d (b-a\mu)(\mu-1) P^{2} (\log 2P)^{-2}(1+o(1))+ O\left(PN^{1-d\gamma_{d,k}+\varepsilon/2}\right) \quad \mbox{if } P\mu\le N.
$$
Combing this with \eqref{mudef}, \eqref{etadef}, \eqref{nudef}, \eqref{key} and \eqref{alphadef}, we obtain
\begin{equation} \label{near}
\begin{split}
\sum\limits_{P\le p< \mu P} \lambda(\mathcal{B}_p)\ge & \frac{(b-a)^2}{4a}\cdot \frac{\min\left(2,c_1,...,c_d\right)^{d-1}}{2^d} \cdot (\mu P)^{-1-(d+1)(\gamma_{d,k}-\varepsilon)} P^2 (\log 2P)^{-2}(1+o(1))+\\
& O\left(P^{-\gamma_{d,k}}N^{1-d\gamma_{d,k}+3\varepsilon/2}\right) \quad \mbox{if } P\mu\le N.
\end{split}
\end{equation}

By splitting the interval $[1,N)$ into intervals of the form $[P,\mu P)$ and summing up, it now follows from \eqref{integral} and \eqref{near} that
\begin{equation*}
\begin{split}
& \int\limits_{a}^{b} F_N(\alpha)d\alpha \\
\ge & \frac{(b-a)^2}{4a}\cdot \frac{\min(2,c_1,...,c_d)^{d-1}}{2^d} \cdot  \left( \sum\limits_{k=0}^{\infty} 
\left(\frac{N}{\mu^k}\right)^{-1-(d+1)(\gamma_{d,k}-\varepsilon)} \left(\frac{N}{\mu^{k+1}}\right)^2 (\log N)^{-2}\right)(1+o(1))\\
= & \frac{(b-a)^2}{4a}\cdot \frac{\min(2,c_1,...,c_d)^{d-1}}{2^d} \cdot \mu^{-2} \cdot \frac{1}{1-\mu^{-(1-(d+1)(\gamma_{d,k}-\varepsilon))}} \cdot 
N^{1-(d+1)(\gamma_{d,k}-\varepsilon)} (\log N)^{-2}(1+o(1)).
\end{split}
\end{equation*}
Further, since $\mu>1$, we have 
$$
1-\mu^{-(1-(d+1)(\gamma_{d,k}-\varepsilon))}\le \left(1-(d+1)(\gamma_{d,k}-\varepsilon)\right)(\mu-1)= \left(1-(d+1)(\gamma_{d,k}-\varepsilon)\right)\cdot \frac{b-a}{2a}.
$$
Hence, we deduce that 
\begin{equation}
\begin{split}
& \int\limits_{a}^{b} F_N(\alpha)d\alpha\\ \ge & (b-a) \cdot \frac{2a^2}{(a+b)^2} \cdot \frac{1}{1-(d+1)(\gamma_{d,k}-\varepsilon)}\cdot \frac{\min\left(c_1,...,c_d,2\right)^{d-1}}{2^d} \times\\
&  N^{1-(d+1)(\gamma_{d,k}-\varepsilon)}
(\log N)^{-2}(1+o(1))\\ \ge & 
(b-a) \cdot \frac{A^2}{B^2} \cdot  \frac{\min\left(c_1,...,c_d,2\right)^{d-1}}{2^{d+1}}  \cdot N^{1-(d+1)(\gamma_{d,k}-\varepsilon)}
(\log N)^{-2}(1+o(1)), 
\end{split}
\end{equation}
establishing the claim of Theorem \ref{Theo}(i).  It remains to prove \eqref{TP}.

\subsection{Reduction to exponential sums}
For $x\in \mathbb{R}$ let
$$
\psi(x):=x-[x]-\frac{1}{2}.
$$
Then we may write $T(P)$ in the form
\begin{equation}
\begin{split}
T(P) & =\sum\limits_{Pa\mu \le n \le bP}\Lambda(n) \prod\limits_{i=1}^d (\delta-(\psi(-c_in)-\psi(-(c_in+\delta)))) \\
& = \sum\limits_{\mathcal{A} \subseteq \{1,...,d\}} \delta^{d-|\mathcal{A}|} T_{\mathcal{A}}(P),  \end{split}
\end{equation}
where
$$
T_{\mathcal{A}}(P):=\sum\limits_{Pa\mu \le n \le bP}\Lambda(n)\prod\limits_{i\in \mathcal{A}} (\psi(-(c_in+\delta))-\psi(-c_in)).
$$
By the prime number theorem,
$$
T_{\mathcal{\emptyset}}(P)=\delta^d \sum\limits_{Pa\mu \le n \le bP} \Lambda(n) \sim \delta^d(b-a\mu)P \quad \mbox{as } P\rightarrow \infty.
$$
Hence, to establish \eqref{TP}, it suffices to prove that for any fixed $\varepsilon>0$ a bound of the form
\begin{equation} \label{error}
T_{\mathcal{A}}(P)=O\left(N^{1-d\gamma_{d,k}+\varepsilon/2}\right) \quad \mbox{if } P\mu\le N
\end{equation}
for all non-empty subsets $\mathcal{A}$ of $\{1,...,d\}$ holds. We reduce the left-hand side to exponential sums, using the following Fourier analytic tool developed by Vaaler \cite{Vaal}.

\begin{lemma}[Vaaler] \label{Vaaler} For $0<|t|<1$ let
$$
W(t)=\pi t(1-|t|) \cot \pi t +|t|.
$$
Fix a positive integer $J$. For $x\in \mathbb{R}$ define 
$$
\psi^{\ast}(x):=-\sum\limits_{1\le |j|\le J} (2\pi i j)^{-1}W\left(\frac{j}{J+1}\right)e(jx)
$$
and
$$
\tau(x):=\frac{1}{2J+2} \sum\limits_{|j|\le J} \left(1-\frac{|j|}{J+1}\right)e(jx).
$$
Then $\tau(x)$ is non-negative, and we have 
$$
|\psi^{\ast}(x)-\psi(x)|\le \tau(x)
$$
for all real numbers $x$. 
\end{lemma}

\begin{proof}
This is Theorem A6 in \cite{GrKo} and has its origin in \cite{Vaal}.
\end{proof}

We set
$$
\tau^{\ast}(x):=\psi(x)-\psi^{\ast}(x).
$$
Then
\begin{equation} \label{thestep}
\begin{split}
T_{\mathcal{A}}(P) & =\sum\limits_{\mathcal{B}\subseteq \mathcal{A}} \sum\limits_{Pa\mu \le n \le bP}\Lambda(n) \left(\prod\limits_{i\in \mathcal{B}} (\psi^{\ast}(-(c_in+\delta))-\psi^{\ast}(-c_in))\right)\times\\ & \left(\prod\limits_{j\in \mathcal{A}\setminus \mathcal{B}} (\tau^{\ast}(-(c_jn+\delta))-\tau^{\ast}(-c_jn))\right)\\
& = U_{\mathcal{A}}(P) + O\left(V_{\mathcal{A}}(P)\right),
\end{split}
\end{equation}
where
$$
U_{\mathcal{A}}(P) := \sum\limits_{Pa\mu \le n \le bP}\Lambda(n) \left(\prod\limits_{i\in \mathcal{A}} (\psi^{\ast}(-(c_in+\delta))-\psi^{\ast}(-c_in))\right)
$$
and
$$ 
V_{\mathcal{A}}(P) := (\log 2P)\sum\limits_{i\in \mathcal{A}} \sum\limits_{Pa\mu \le n \le bP}\left(\tau(-(c_in+\delta))+\tau(-c_in))\right),
$$
where the $O$-term $V_{\mathcal{A}}(P)$ arrives by using $|\psi^{\ast}(x)|\le 2$, $|\tau^{\ast}(x)|\le \tau(x)\le 1$, $\Lambda(n)\le \log n$  and the triangle inequality.

The definition of the function $\tau(x)$ gives
$$
V_{\mathcal{A}}(P)=\frac{\log 2P}{2J+2} \sum\limits_{i\in \mathcal{A}} \sum\limits_{|j|\le J} \left(1-\frac{|j|}{J+1}\right) (1+e(j\delta)) \sum\limits_{Pa\mu \le n \le bP} e(jc_in),
$$
and the definition of $\psi^{\ast}(x)$ gives, after multiplying out and re-arranging summations, 
\begin{equation*}
\begin{split}
U_{\mathcal{A}}(P) & =- \sum\limits_{1\le |j_1|\le J} \cdots \sum\limits_{1\le |j_h|\le J} \left(\prod\limits_{k=1}^{h} \left((2\pi i j_k)^{-1}W\left(\frac{j_k}{J+1}\right)
(1+e(j_k\delta)) \right)\right)\times \\ & \sum\limits_{Pa\mu \le n \le bP}\Lambda(n) e\left(n \sum\limits_{k\in \mathcal{A}} j_kc_{l_k}\right),
\end{split}
\end{equation*}
where we suppose that $J$ is a positive integer satisfying $J\le P$ and 
$$
h:=|\mathcal{A}| \quad \mbox{and} \quad \mathcal{A}=\left\{l_1,...,l_h\right\}.
$$

We further estimate $V_{\mathcal{A}}(P)$ by
\begin{equation} \label{V}
\begin{split}
V_{\mathcal{A}}(P) \ll & \frac{P\log 2P}{J} +
\frac{\log 2P}{J} \cdot \sum\limits_{i=1}^d \sum\limits_{1\le j\le J} \left| \sum\limits_{Pa\mu \le n \le bP} e(jc_in) \right|\\
\ll & \frac{P\log 2P}{J} +
\frac{\log 2P}{J} \cdot \sum\limits_{i=1}^d \sum\limits_{1\le j\le J} \min\left(P,\left|\left|jc_i\right|\right|^{-1}\right) \\
\ll & \frac{P\log 2P}{J} +
(\log 2P)\sum\limits_{i=1}^d \sum\limits_{1\le j\le J} \min\left(\frac{P}{j},\left|\left|jc_i\right|\right|^{-1}\right) \\ 
=: & \frac{P\log 2P}{J} + (\log 2P){\tilde V}_d(P),
\end{split}
\end{equation}
where the term $P(\log 2P)/J$ bounds the contribution of $j=0$, 
and we estimate $U_{\mathcal{A}}(P)$ by
\begin{equation} \label{U}
U_{\mathcal{A}}(P) \ll {\tilde U}_{\mathcal{A}}(P) := 
\sum\limits_{1 \le |j_1|\le J} \cdots \sum\limits_{1\le |j_h| \le J} \frac{1}{\left|j_1\cdots j_h\right|} \left|
\sum\limits_{Pa\mu \le n \le bP}\Lambda(n) e\left(n \sum\limits_{k\in \mathcal{A}} j_kc_{l_k}\right) \right|.
\end{equation}
It remains to estimate ${\tilde U}_{\mathcal{A}}(P)$ and ${\tilde V}_{d}(P)$. The estimation of ${\tilde V}_d(P)$ is clearly easier than that of ${\tilde U}_{\mathcal{A}}(P)$. We first deal with the term ${\tilde U}_{\mathcal{A}}(P)$.

\subsection{Application of Vaughan's identity}
We convert the inner sum involving the von Mangoldt function on the right-hand side of \eqref{U}  into bilinear sums using Vaughan's identity.

\begin{lemma}[Vaughan] \label{Vaughan}
Let $u\ge 1$, $v\ge 1$, $uv\le x$. Then we have for every arithmetic function $f: \mathbb{N}\rightarrow \mathbb{C}$ the estimate
$$
\sum\limits_{u<n\le x} f(n)\Lambda(n) \ll (\log 2x)T_1+T_2
$$
with
$$
T_1:= \sum\limits_{l\le uv} \max\limits_{w} \left| \sum\limits_{w\le m\le x/l} f(ml) \right|
$$
and
$$
T_2:=\left| \sum\limits_{u<m\le x/v} \sum\limits_{v<l\le x/m} \Lambda(m)b(l) f(ml) \right|,
$$
where $b(l)$ is an arithmetic function which only depends on $v$ and satisfies the inequality $b(l)\le \tau(l)$, $\tau(l)$ being the number of divisors of $l$.
\end{lemma}

\begin{proof}
This is Satz 6.1.2. in \cite{Brud} and has its origin in \cite{Vaug}.
\end{proof}

We use Lemma \ref{Vaughan} with parameters $u$ and $v$ satisfying $1\le u=v\le (Pa\mu)^{1/2}$, to be fixed later, $x:=bP$ and 
$$
f(n):=\begin{cases} e\left(nc\right) & \mbox{ if } Pa\mu \le n\le bP, \\ 0 & \mbox{ if } n<Pa\mu \end{cases}
$$
with
$$
c:=\sum\limits_{k\in \mathcal{A}} j_kc_{l_k}
$$
to deduce that
\begin{equation} \label{ZHest}
{\tilde U}_{\mathcal{A}}(P)\ll (\log 2P)Z_1+Z_2,
\end{equation}
where
$$
Z_1:=\sum\limits_{1\le |j_1|\le J}\cdots \sum\limits_{1\le |j_h|\le J}  \frac{1}{\left|j_1\cdots j_h\right|} \sum\limits_{l\le u^2} \max\limits_{Pa\mu/l \le w\le bP/l} \left| \sum\limits_{w\le m\le bP/l} 
e\left(ml \sum\limits_{k\in \mathcal{A}} j_kc_{l_k}\right) \right|
$$
and 
$$
Z_2:=\sum\limits_{1\le |j_1|\le J}\cdots \sum\limits_{1\le |j_h|\le J} \frac{1}{\left|j_1\cdots j_h\right|} \left| \sum\limits_{u<m\le bP/u} \ \sum\limits_{\max(u,Pa\mu/m)\le l\le bP/m} \Lambda(m)b(l)  
e\left(ml \sum\limits_{k\in \mathcal{A}} j_kc_{l_k}\right) \right|.
$$

Obviously,
\begin{equation} \label{Z1H}
Z_1\ll  \sum\limits_{1\le |j_1|\le J}\cdots \sum\limits_{1\le |j_h|\le J} \frac{1}{\left|j_1\cdots j_h\right|}\sum\limits_{l\le u^2} \min\left( \frac{P}{l}, \left|\left| l \sum\limits_{k\in \mathcal{A}} j_kc_{l_k} \right|\right|^{-1}\right).
\end{equation}
We shall boil down $Z_2$ to similar terms. Rearranging the summation gives 
$$
Z_2=\sum\limits_{1\le |j_1|\le J}\cdots \sum\limits_{1\le |j_h|\le J} \frac{1}{\left|j_1\cdots j_h\right|} \left| \sum\limits_{a\mu u/b\le l\le bP/u} b(l) \sum\limits_{M_1(l) \le m\le M_2(l)}  \Lambda(m)  e\left(ml \sum\limits_{k\in \mathcal{A}} j_kc_{l_k}\right)\right|
$$
with 
$$
M_1(l):=\max([u]+1,Pa\mu/l) \quad \mbox{and} \quad M_2(l):=bP/l.
$$
We now observe that
\begin{equation} \label{Z2HK}
Z_2\ll (\log 2P) \max\limits_{u\le L\le bP/u} Z_2(L),
\end{equation}
where
\begin{equation*}
Z_2(L) :=  \sum\limits_{1\le |j_1|\le J}\cdots \sum\limits_{1\le |j_h|\le J} \frac{1}{\left|j_1\cdots j_h\right|} \sum\limits_{L\le l\le 2L} b(l) \cdot 
\left| \sum\limits_{M_1(l)\le m\le M_2(l)} \Lambda(m) e\left(ml \sum\limits_{k\in \mathcal{A}} j_kc_{l_k}\right) \right|. 
\end{equation*}
Recall that $1\le J\le P$. Using the Cauchy-Schwarz inequality and the bound
$$
\sum\limits_{L\le l\le 2L} |b(l)|^2 \le \sum\limits_{L\le l\le 2L} d(l)^2 \ll L(\log 2L)^3, 
$$ 
$d(l)$ being the divisor function, 
and expanding the square, we obtain
\begin{equation*} 
\begin{split}
Z_2(L)^2 \ll &  \left(\sum\limits_{L\le l\le 2L} |b(l)|^2\right)   \left(\sum\limits_{1\le |j_1|\le J}\cdots \sum\limits_{1\le |j_h|\le J} \frac{1}{\left|j_1\cdots j_h\right|}\right) \times\\ & \sum\limits_{1\le |j_1|\le J}\cdots \sum\limits_{1\le |j_h|\le J} \frac{1}{\left|j_1\cdots j_h\right|} \cdot
 \sum\limits_{L\le l\le 2L} \left| \sum\limits_{M_1(l)\le m\le M_2(l)} \Lambda(m)  e\left(ml \sum\limits_{k\in \mathcal{A}} j_kc_{l_k}\right) \right|^2\\
\ll & \left(\sum\limits_{L\le l\le 2L} |b(l)|^2\right)  (\log 2P)^{h} \sum\limits_{1\le |j_1|\le J}\cdots \sum\limits_{1\le |j_h|\le J} \frac{1}{\left|j_1\cdots j_h\right|} \times \\ &
 \sum\limits_{L\le l\le 2L} \left| \sum\limits_{M_1(l)\le m\le M_2(l)} \Lambda(m)  e\left(ml \sum\limits_{k\in \mathcal{A}} j_kc_{l_k}\right) \right|^2.
\end{split}
\end{equation*}
Expanding the square, and splitting the resulting expression into a diagonal and non-diagonal term, we estimate the above further by
\begin{equation*} 
\begin{split}
Z_2(L)^2 \ll & L(\log 2P)^{h+3}   \sum\limits_{1\le |j_1|\le J}\cdots \sum\limits_{1\le |j_h|\le J} \frac{1}{\left|j_1\cdots j_h\right|} \sum\limits_{L\le l\le 2L}  \sum\limits_{M_1(l)\le m\le M_2(l)} \Lambda(m)^2 + \\ &
L(\log 2P)^{h+3} \sum\limits_{1\le |j_1|\le J}\cdots \sum\limits_{1\le |j_h|\le J} \frac{1}{\left|j_1\cdots j_h\right|} \times\\ & 
\left| \sum\limits_{L\le l\le 2L}  \sum\limits_{M_1(l)\le m_1<m_2\le M_2(l)} \Lambda(m_1)\Lambda(m_2) e\left((m_2-m_1)l \sum\limits_{k\in \mathcal{A}} j_kc_{l_k}\right)\right|.
\end{split}
\end{equation*}
Exchanging summation and estimating geometric sums, we deduce that
\begin{equation} \label{aftercauchy}
\begin{split}
Z_2(L)^2 \ll & LP(\log 2P)^{2h+5} + L(\log 2P)^{h+3}\sum\limits_{1\le |j_1|\le J}\cdots \sum\limits_{1\le |j_h|\le J} \frac{1}{\left|j_1\cdots j_h\right|} \times \\ & \sum\limits_{u\le m_1<m_2\le bP/L} \left| \sum\limits_{\max(L,Pa\mu/m_1)\le l\le \min(2L,bP/m_2)} e\left((m_2-m_1)l \sum\limits_{k\in \mathcal{A}} j_kc_{l_k}\right)\right|\\
\ll & LP(\log 2P)^{2h+5}+  L(\log 2P)^{h+3} \sum\limits_{1\le |j_1|\le J}\cdots \sum\limits_{1\le |j_h|\le J} \frac{1}{\left|j_1\cdots j_h\right|}\times\\ & \sum\limits_{u\le m_1<m_2\le bP/L} \min\left(L, \left|\left| (m_2-m_1)\sum\limits_{k\in \mathcal{A}} j_kc_{l_k}  \right|\right|^{-1}\right)\\
\ll & LP(\log 2P)^{2h+5}+ P(\log 2P)^{h+3}  \sum\limits_{1\le |j_1|\le J}\cdots \sum\limits_{1\le |j_h|\le J} \frac{1}{\left|j_1\cdots j_h\right|} \times\\ &  \sum\limits_{1\le m \le bP/L} 
\min\left(\frac{P}{m},\left|\left| m\sum\limits_{k\in \mathcal{A}} j_kc_{l_k}  \right|\right|^{-1}\right). 
\end{split}
\end{equation}

\subsection{Construction of a Dirichlet approximation}
The treatments of $Z_1$ and $Z_2$ lead to sums of the form 
\begin{equation} \label{R}
R_{\mathcal{A}}(M,x):= \sum\limits_{1\le |j_1|\le J}\cdots \sum\limits_{1\le |j_h|\le J} \frac{1}{\left|j_1\cdots j_h\right|} \sum\limits_{1\le m \le M} \min\left(\frac{x}{m},
\left|\left| m\sum\limits_{k\in \mathcal{A}} j_kc_{l_k}  \right|\right|^{-1}\right)
\end{equation}
which we estimate in the following. 
Here the Diophantine properties of the vector $(c_1,...,c_d)$ come into play. 

Splitting each of the $j_k$-intervals into $O(\log P)$ dyadic intervals and summing up their contributions, we obtain
\begin{equation} \label{Rsplit}
R_{\mathcal{A}}(M,x)\ll (\log 2P)^d \sup\limits_{1\le H_1,...,H_d\le J } \frac{R_d(H_1,...,H_d,M,x)}{H_1\cdots H_d}, 
\end{equation}
where 
\begin{equation}\label{RH}
R_d(H_1,...,H_d,M,x):=  \sum\limits_{\substack{|j_1|\le H_1,...,|j_d|\le H_d\\ {\bf j}\not={\bf 0}}} \sum\limits_{1\le m \le M} \min\left(\frac{x}{m},
\left|\left| m {\bf j}\cdot {\bf c}  \right|\right|^{-1}\right)
\end{equation}
with
$$
{\bf j}:=(j_1,...,j_d) \quad \mbox{and} \quad {\bf c}:=(c_1,...,c_d).
$$

The key point is now to approximate ${\bf j} \cdot {\bf c}$ by rational numbers, using Dirichlet's approximation theorem, where the denominators are uniformly bounded from below. Let ${\bf j}\in \mathbb{Z}^d\setminus{{\bf 0}}$ and $X$ be a positive integer, to be fixed later. By Dirichlet's approximation theorem, there exist integers $a,q$ such that gcd$(a,q)=1$, $1\le q\le X$ and 
$$
\left|{\bf j}\cdot {\bf c} - \frac{a}{q}\right| \le \frac{1}{qX}. 
$$
Hence,
$$
\left|\left| q{\bf j}\cdot {\bf c} \right|\right| \le \frac{1}{X}.
$$
On the other hand, by condition \eqref{condi} in Theorem \ref{super}, we have
$$
\frac{C}{q^k||{\bf j}||_{\infty}^k} < \left|\left| q{\bf j}\cdot {\bf c} \right|\right|.
$$
It follows that 
$$
q>\frac{(CX)^{1/k}}{||{\bf j}||_{\infty}}.
$$

Now we apply the following well-known lemma. 

\begin{lemma} \label{standard} Let $L\ge 1$ and $x>1$. Suppose that $|c-a/q|\le q^{-2}$ with $a\in \mathbb{Z}$, $q\in \mathbb{N}$ and $(a,q)=1$. Then
$$
\sum\limits_{1\le l \le L} \min\left(\frac{x}{l},\left|\left| lc  \right|\right|^{-1}\right)\ll \left(\frac{x}{q}+L+q\right)(\log 2Lqx).
$$
\end{lemma}

\begin{proof} This is Lemma 6.4.4. in \cite{Brud}.
\end{proof} 

It follows that 
\begin{equation*} 
\sum\limits_{1\le m \le M} \min\left(\frac{x}{m},\left|\left| m {\bf j}\cdot {\bf c}  \right|\right|^{-1}\right) \ll \left(\frac{x||{\bf j}||_{\infty}}{X^{1/k}}+M+X\right)(\log 2MXx)
\end{equation*}
and hence
\begin{equation} \label{Rsplitesti}
R_d(H_1,...,H_d,M,x) \ll H_1\cdots H_d  \left(\frac{x\max(H_1,...,H_d)}{X^{1/k}}+M+X\right)(\log 2MXx).
\end{equation}
Combining \eqref{Rsplit} and \eqref{Rsplitesti}, we obtain
\begin{equation*} 
R_{\mathcal{A}}(M,x)\ll (\log 2P)^{d+1} \left(\frac{xJ}{X^{1/k}}+M+X\right).
\end{equation*}
Now choosing
$$
X:=\left[(xJ)^{1-1/(k+1)}\right],
$$
we deduce that
\begin{equation} \label{RAest}
R_{\mathcal{A}}(M,x)\ll (\log 2P)^{d+1} \left(M+(xJ)^{1-1/(k+1)}\right).
\end{equation}

\subsection{Completion of the proof} 
Using \eqref{Z1H}, \eqref{aftercauchy}, \eqref{Rsplit} and \eqref{RAest}, we get
\begin{equation} \label{Z1esti}
Z_1\ll (\log 2P)^{d+1}\left(u^2+(PJ)^{1-1/(k+1)}\right)
\end{equation}
and 
\begin{equation*}
Z_2(L)\ll (\log 2P)^{3d/2+5/2}\left(P^{1/2}L^{1/2} + PL^{-1/2}+P^{1-1/(2(k+1))}J^{1/2-1/(2(k+1))}\right)
\end{equation*}
and hence by (\ref{Z2HK}),
\begin{equation}\label{Z2esti}
Z_2\ll (\log 2P)^{3d/2+7/2} \left(Pu^{-1/2}+P^{1-1/(2(k+1)}J^{1/2-1/(2(k+1))}\right).
\end{equation}
Combining \eqref{U}, \eqref{ZHest}, \eqref{Z1esti} and \eqref{Z2esti}, we get
\begin{equation*}
\tilde{U}_{\mathcal{A}}(P)\ll P^{\varepsilon}\left(u^2+Pu^{-1/2}+(PJ)^{1-1/(k+1)}+P^{1-1/(2(k+1))}J^{1/2-1/(2(k+1))}\right).
\end{equation*}  
Now choosing $u:=P^{2/5}$, it follows that
\begin{equation} \label{UAesti}
\tilde{U}_{\mathcal{A}}(P)\ll P^{\varepsilon}\left(P^{4/5}+(PJ)^{1-1/(k+1)}+P^{1-1/(2(k+1))}J^{1/2-1/(2(k+1))}\right).
\end{equation}

We now turn to the term $\tilde{V}_{d}(P)$, defined in \eqref{V}. By \eqref{RH} and \eqref{Rsplitesti}, we have
$$
\tilde{V}_{d}(P)\ll R_d(1,...,1,J,P) \ll \left(\frac{P}{X^{1/k}}+J+X\right)(\log 2JXP)
$$
for any positive integer $X$.
Choosing 
$$
X:=\left[P^{1-1/(k+1)}\right]
$$
gives
$$
\tilde{V}_{d}(P)\ll \log(2P)\left(J+P^{1-1/(k+1)}\right).
$$
Combining this with \eqref{V}, we find
\begin{equation} \label{VAesti}
V_{\mathcal{A}}(P)\ll (\log 2P) \left(PJ^{-1}+J+P^{1-1/(k+1)}\right).
\end{equation}

Now putting \eqref{thestep}, \eqref{UAesti} and \eqref{VAesti} together, we arrive at
$$
T_{\mathcal{A}}(P)\ll P^{\varepsilon}\left(PJ^{-1}+J+P^{4/5}+(PJ)^{1-1/(k+1)}+P^{1-1/(2(k+1))}J^{1/2-1/(2(k+1))}\right).
$$
Now choosing 
\begin{equation} \label{Jchoice}
J:=\left[P^{1/(3k+2)}\right],
\end{equation}
we get
$$
T_{\mathcal{A}}(P)\ll P^{1-1/(3k+2)+\varepsilon},
$$
which proves \eqref{error} upon replacing $\varepsilon$ by $\varepsilon/2$. This completes the proof of Theorem \ref{Theo}(i). 

\section{Proof of Theorem \ref{Theo}(ii)} 

\subsection{Sieve theoretical approach}
We are broadly following the treatment in \cite{HaJo} with appropriate modifications because the linear case, considered here, requires a different treatment. In particular, as in the previous section, the Diophantine properties of the vector $(c_1,...,c_d)$ will come into play. Let $\{\cdot\}$ represent the fractional part, and put
$$
\mu:=N^{\varepsilon-\gamma_{d,k}}.
$$
Write
$$
\mathcal{A}=\mathcal{A}(\alpha)=\{n[n\alpha]\ :\  1\le n\le N, \ \{n\alpha\}<\mu, \ \{nc_i\alpha\}<\mu \mbox{ for } i=1,...,d\}.
$$
We desire to show that $\mathcal{A}$ does not contain too many products of two primes. To this end, we apply a two-dimensional upper bound sieve (see \cite{HaRi}, Theorem 5.2). We therefore need to obtain an asymptotic formula for the number of solutions to 
$$
n[n\alpha]\equiv 0 \bmod{q}, \ 1\le n\le N,
$$
with 
\begin{equation} \label{max}
\{n\alpha\}<\mu \mbox{ and } \{nc_i\alpha\}<\mu \mbox{ for } i=1,...,d,
\end{equation}
where 
$$
q\le Q:=N^{\varepsilon}.
$$
For this it suffices to establish a formula for the number of solutions to
$$
n \equiv 0 \bmod{t_1}, \quad [n\alpha]\equiv 0 \bmod{t_2}
$$
subject to \eqref{max}. We can combine \eqref{max} with the congruence conditions to require
\begin{equation} \label{threeconds}
1\le n\le \frac{N}{t_1}, \quad \left\{\frac{nt_1\alpha}{t_2}\right\}<\frac{\mu}{t_2}, \quad \left\{nt_1c_i\alpha\right\}<\mu \mbox{ for } i=1,...,d,
\end{equation}
and count the number $S(\alpha;t_1,t_2)$ of solutions to \eqref{threeconds} using Fourier analysis. To this end, we write 
$$
S(\alpha,t_1,t_2)=\sum\limits_{1\le n\le N/t_1} \left(\left[\frac{nt_1\alpha}{t_2}\right]-\left[\frac{nt_1\alpha}{t_2}-\frac{\mu}{t_2}\right]\right)\cdot \prod\limits_{i=1}^d \left(
[nt_1c_i\alpha]-[nt_1c_i\alpha-\mu]\right)
$$
and evaluate this term in a similar way as the term $T(P)$ defined in \eqref{TPdef} using Vaaler's Lemma 3.1. By a chain of similar calculations, we arrive
at the asymptotic estimate
\begin{equation} \label{arriv} 
S(\alpha;t_1,t_2)=\frac{N\mu^{d+1}}{t_1t_2}+O\left(\frac{N\mu^{d}}{L}+E(\alpha;t_1,t_2)\right),
\end{equation}
where we set 
$$
L:=Q^3\mu^{-1}
$$
and 
$$
E(\alpha;t_1,t_2):=\frac{\mu^{d+1}}{t_2}\sum\limits_{\substack{|m_0|\le L,...,|m_d|\le L\\ (m_0,...,m_d)\not=(0,...,0)}} \left| \sum\limits_{1\le n\le N/t_1}
e\left(n\alpha t_1\left(\frac{m_0}{t_2}+\sum\limits_{i=1}^d c_im_i\right)\right)\right|.
$$
The above estimate \eqref{arriv} is analog to the equation after (26) in \cite{HaJo}.

Now, applying the upper bound sieve gives
\begin{equation}
F_N(\alpha)\le \frac{C(\varepsilon)N\mu^{d}}{\log^2 N} + O\left(J_N(\alpha)\right),
\end{equation}
where 
\begin{equation*}
\begin{split}
J_N(\alpha):= & \sum\limits_{t_1t_2\le Q} (t_1t_2)^{\varepsilon}\left(\frac{N\mu^{d}}{L}+E(\alpha;t_1,t_2)\right)\\ = & \sum\limits_{t_1t_2\le Q} (t_1t_2)^{\varepsilon}E(\alpha;t_1,t_2)+o\left(\frac{N\mu^{d+1}}{\log^2 N}\right) \quad \mbox{as } N\rightarrow \infty.
\end{split}
\end{equation*}
Hence, to establish the claim in Theorem \ref{Theo}(i), it suffices to show that
\begin{equation} \label{average}
\sum\limits_{t_1t_2\le Q} (t_1t_2)^{\varepsilon}\int\limits_A^B E(\alpha;t_1,t_2)d\alpha = o\left(\frac{N\mu^{d+1}}{\log^2 N}\right) \quad \mbox{as } N\rightarrow \infty, \ N\in \mathcal{S}.
\end{equation}

\subsection{Average estimation for $E(\alpha;	t_1,t_2)$} 
To estimate the expression on the right-hand side of \eqref{average}, we first observe that
\begin{equation} \label{E2alpha}
E(\alpha;t_1,t_2)\ll \frac{\mu^{d+1}}{t_2} \sum\limits_{\substack{|m_0|\le L,...,|m_d|\le L\\ (m_0,...,m_d)\not=(0,...,0)}} \min\left(\frac{N}{t_1},\left|\left| \alpha t_1\left(\frac{m_0}{t_2}+m_1c_1+...+m_dc_d\right) \right|\right|^{-1}\right)
\end{equation}
and note that if $(m_0,...,m_d)\not=(0,...,0)$, then the term 
$$
 t_1\left(\frac{m_0}{t_2}+m_1c_1+...+m_dc_d\right) 
$$
on the right-hand side of \eqref{E2alpha} is non-zero because $(c_1,...,c_d)$ is $k$-Diophantine. For an estimation of the integral in \eqref{average}, we now use the following lemma.

\begin{lemma} \label{intlemma}
Suppose that $0<A<B$, $K\ge 2$ and $x\not=0$. Then
\begin{equation} \label{intesti}
\int\limits_A^B \min\left(K,\left|\left|\alpha x \right|\right|^{-1}\right)d\alpha=O_{A,B}\left(\min\left\{K,\max\left\{1,|x|^{-1}\right\}\right\}\log K\right).
\end{equation}
\end{lemma}

\begin{proof} We confine ourselves to the case when $x>0$ since the case when $x<0$ is similar. By change of variables $\beta=\alpha x$, we get
\begin{equation} \label{change}
\int\limits_A^B \min\left(K,\left|\left| \alpha x \right|\right|^{-1}\right)d\alpha=\frac{1}{x}\int\limits_{xA}^{xB}  \min\left(K,\left|\left|\beta\right|\right|^{-1}\right)d\beta.
\end{equation}
By periodicity of the integrand, if $x(B-A)\ge 1$, we have
\begin{equation} \label{case1}
\frac{1}{x} \int\limits_{xA}^{xB}  \min\left(K,\left|\left|\beta\right|\right|^{-1}\right)d\beta\le \frac{[x(B-A)+1]}{x} \int\limits_0^1 \min\left(K,\left|\left|\beta\right|\right|^{-1}\right)d\beta =O_{A,B}\left(\log K\right).
\end{equation}
If $1/K\le x(B-A)<1$, then
\begin{equation} \label{case2}
\frac{1}{x} \int\limits_{xA}^{xB}  \min\left(K,\left|\left|\beta\right|\right|^{-1}\right)d\beta\le \frac{1}{x}\int\limits_{-x(B-A)/2}^{x(B-A)/2} \min\left(K,\left|\left|\beta\right|\right|^{-1}\right)d\beta =O_{A,B} \left(\frac{\log K}{x}\right). 
\end{equation}
If $0<x(B-A)< 1/K$, then trivially
\begin{equation} \label{case3}
\frac{1}{x}\int\limits_{xA}^{xB} \min\left(K,\left|\left|\beta\right|\right|^{-1}\right)d\beta=O_{A,B}(K).
\end{equation}
Combining \eqref{change}, \eqref{case1}, \eqref{case2} and \eqref{case3}, we deduce the claim when $x>0$, which completes the proof. 
\end{proof}

Now, employing \eqref{E2alpha} and Lemma \ref{intlemma}, we get
\begin{equation*} 
\begin{split}
& \int\limits_A^B E(\alpha;t_1,t_2)d\alpha\\ \ll & \frac{\mu^{d+1}}{t_2} \sum\limits_{\substack{|m_0|\le L,...,|m_d|\le L\\ (m_0,...,m_d)\not=(0,...,0)}} \min\left\{\frac{N}{t_1},\max\left\{1,
\left| t_1\left(\frac{m_0}{t_2}+m_1c_1+...+m_dc_d\right)\right|^{-1}\right\}\log N\right\} \\
\ll & \mu^{d+1}(\log N) \sum\limits_{\substack{|m_0|\le L,...,|m_d|\le L\\ (m_0,...,m_d)\not=(0,...,0)}} \min\left\{\frac{N}{t_1t_2},\max\left\{\frac{1}{t_2},
\left| t_1m_0+t_1t_2\left(m_1c_1+...+m_dc_d\right)\right|^{-1}\right\}\right\}. 
\end{split}
\end{equation*}
The contribution of $(m_1,...m_d)=(0,...,0)$ to the last line is bounded by $O\left(\mu^{d+1}L\log N\right)$, and the contribution of $(m_1,...,m_d)\not=0$ is bounded by 
\begin{equation*}
\begin{split}
\ll & \mu^{d+1}(\log N) \sum\limits_{\substack{|m_1|\le L,...,|m_d|\le L\\ (m_1,...,m_d)\not=(0,...,0)}} \sum\limits_{|m_0|\le L} \min\left\{\frac{N}{t_1t_2},\max\left\{\frac{1}{t_2},
\left| t_1m_0+t_1t_2\left(m_1c_1+...+m_dc_d\right)\right|^{-1}\right\}\right\}\\
\ll & \mu^{d+1}(\log N)\sum\limits_{\substack{|m_1|\le L,...,|m_d|\le L\\ (m_1,...,m_d)\not=(0,...,0)}} \min\left\{\frac{N}{t_1t_2},
\left|\left| t_1t_2\left(m_1c_1+...+m_dc_d\right)\right|\right|^{-1}\right\} +\mu^{d+1}L^{d+1}\log N.
\end{split}
\end{equation*}
So altogether, we have 
 \begin{equation*} 
\begin{split}
& \int\limits_A^B E(\alpha;t_1,t_2)d\alpha \ll \mu^{d+1}L^{d+1}\log N+\\
& \mu^{d+1}(\log N)\sum\limits_{\substack{|m_1|\le L,...,|m_d|\le L\\ (m_1,...,m_d)\not=(0,...,0)}} \min\left\{\frac{N}{t_1t_2},
\left|\left| t_1t_2\left(m_1c_1+...+m_dc_d\right)\right|\right|^{-1}\right\}.
\end{split}
\end{equation*}

Summing up over $t_1$ and $t_2$, writing $q=t_1t_2$ and using $d(q)=O\left(q^{\varepsilon}\right)$, $d(q)$ being the divisor function, we obtain
\begin{equation*}
\sum\limits_{t_1t_2\le Q} (t_1t_2)^{\varepsilon}\int\limits_A^B E(\alpha;t_1,t_2)d\alpha \\ 
\ll N^{\varepsilon} QL^{d+1}\mu^{d+1} + N^{\varepsilon} \mu^{d+1} R_d(L,...,L,Q,N),
\end{equation*}
where $R_d(L,...,L,Q,N)$ is defined as in \eqref{RH}. Combining this with \eqref{Rsplitesti}, we get
\begin{equation*}
\begin{split}
\sum\limits_{t_1t_2\le Q} (t_1t_2)^{\varepsilon}\int\limits_A^B E(\alpha;t_1,t_2)d\alpha
\ll N^{\varepsilon} L^d \mu^{d+1} \left(L+\frac{LN}{X^{1/k}}+Q+X\right)\log(2QNX)
\end{split}
\end{equation*}
for any positive integer $X$. Choosing $X:=\left[(LN)^{1-1/(k+1)}\right]$, we deduce that
\begin{equation*} 
\sum\limits_{t_1t_2\le Q} (t_1t_2)^{\varepsilon}\int\limits_A^B E(\alpha;t_1,t_2)d\alpha \ll  N^{\varepsilon} 
L^{d}\mu^{d+1} \left(L+Q+(LN)^{1-1/(k+1)}\right)
\end{equation*}
which implies \eqref{average} upon 
recalling $\mu:=N^{\varepsilon-\gamma_{d,k}}$, $Q:=N^{\varepsilon}$, $L:=Q^3\mu^{-1}$ and \eqref{gammadef}.

\end{document}